\documentclass[a4paper,11pt]{amsart}
\usepackage{dsfont,fancyhdr,hyperref,mathrsfs,amsmath,amscd,amsthm,amsfonts,latexsym,amssymb,stmaryrd}
\usepackage[all]{xy}

\DeclareMathOperator{\trace}{trace}

\DeclareMathOperator{\Ric}{Ric}
\DeclareMathOperator{\dif}{d}

\renewcommand{\1}{\mathds{1}}

\newcommand{\ol}{\mathcal{O}}

\def \a{\alpha}

\def \b{\beta}

\def \G{\Gamma}
\def \l{\lambda}
\def \o{\omega}

\def \phi{\varphi}
\def \Phi{\varPhi}
\def \p{\pi}
\def \r{\rho}

\def \C{\mathbb{C}\,}

\def\widecheckg{g^{\hspace*{-2.5pt}\vbox to 5pt{\hbox to
0pt{\LARGE$\check{}$}}}\hspace*{2pt}}

\def\widecheckl{\lambda^{\hspace*{-3.5pt}\vbox to 8pt{\hbox to
0pt{\LARGE$\check{}$}}}\hspace*{2pt}}

\begin{document}

\title{Projective structures and $\r$-connections}
\author{Radu Pantilie} 
\thanks{} 
\email{\href{mailto:radu.pantilie@imar.ro}{radu.pantilie@imar.ro}}
\address{R.~Pantilie, Institutul de Matematic\u a ``Simion~Stoilow'' al Academiei Rom\^ane,
C.P. 1-764, 014700, Bucure\c sti, Rom\^ania}

\newtheorem{thm}{Theorem}[section]
\newtheorem{lem}[thm]{Lemma}
\newtheorem{cor}[thm]{Corollary}
\newtheorem{prop}[thm]{Proposition}

\theoremstyle{definition}

\newtheorem{defn}[thm]{Definition}
\newtheorem{rem}[thm]{Remark}
\newtheorem{exm}[thm]{Example}

\numberwithin{equation}{section}

\begin{abstract}
We extend T.~Y.~Thomas's approach to the projective structures, over the complex analytic category, by involving the $\r$-connections. 
This way, a better control of the projective flatness is obtained and, consequently, we have, for example, the following application: 
if the twistor space of a quaternionic manifold $P$ is endowed with a complex projective structure then $P$ can be locally identified, 
through quaternionic diffeomorphisms, with the quaternionic projective space. 
\end{abstract}

\maketitle
\thispagestyle{empty}
\vspace{-4mm} 
\begin{center}
\emph{This paper is dedicated to the 150th anniversary of the Romanian Academy.}
\end{center}

\section*{Introduction} 

\indent 
One of the problems with the Cartan connections approach to the projective structures is that there are much more `constant vector fields', on the corresponding 
principal bundle, than necessary to produce the geodesics. For example, the left invariant vector fields corresponding to nilpotent elements of degree $n\geq3$ 
of $\mathfrak{sl}(n)$ will, also, produce, for example, Veronese curves, on the corresponding projective space.\\ 
\indent 
Fortunately, in \cite{Tho-1926} (see \cite{Rob-proj}\,; see, also, \cite{CraSau-proj} for a nice review of projective structures in the smooth setting) 
it is, essentially, shown that any projective structure on a smooth manifold $M$ corresponds to an invariant Ricci flat torsion free connection 
on $\det(TM)$\,. However, the extension of this approach over the complex analytic category is nontrivial as, in this case, by \cite{At-57}\,, 
the relevant bundles (for example, the tautological line bundle over the complex projective space) can never be endowed with a connection.\\ 
\indent  
Such an extension has been carried over in \cite{MolMor-1996}\,, under the assumption that the canonical line bundle admits an $(n+1)$th root, 
where $n$ is the dimension of the manifold (see, also, \cite{Arm-proj_II} for an extension, of the T.~Y.~Thomas's approach, over odd dimensional complex manifolds).\\ 
\indent  
In this paper, we work out this extension by involving the $\r$-connections introduced in \cite{Pan-qgfs} (see Definition \ref{defn:ro-connection}\,, below). 
The obtained main result (Theorem \ref{thm:proj_ro}\,) then provides a surprisingly simple 
(and improved) characterisation of projective flatness (Corollary \ref{cor:proj_ro_flat}\,). From the applications, we mention, here, only the 
following: if the twistor space of a quaternionic manifold $P$ is endowed with a complex projective structure then $P$ can be locally identified, 
through quaternionic diffeomorphisms, with the quaternionic projective space.\\ 
\indent 
I am grateful to Ben~McKay for informing me about \cite{BisMcK-holo_Cartan} and \cite{MolMor-1996}\,.

\section{Complex projective structures and $\r$-connections} 

\indent 
In this paper, we work in the category of complex manifolds. (The corresponding extensions over the smooth category is easy to be dealt with.)\\ 
\indent 
Recall that two connections on a manifold are \emph{projectively equivalent} if and only if they have the same geodesics (up to parametrizations).\\ 
\indent 
The following two results are well known. For the reader's convenience, we sketch their proofs. 

\begin{prop} 
Let $M$ be a manifold endowed with a connection $\nabla$. Then there exists a torsion free connection on $M$ which is projectively equivalent 
to $\nabla$.  
\end{prop} 
\begin{proof} 
This follows quickly from the following two facts. Firstly, the set of connections on $M$, if nonempty, is an affine space over the space of sections 
of ${\rm End}(TM)$.\\ 
\indent 
Secondly, if $\bigl(\G^i_{jk}\dif\!x^k\bigr)_{i,j}$ are the local connection forms of a connection $\nabla$ on $M$, 
with respect to local charts $\bigl(x^i\bigr)_i$ on $M$  (that is, $\nabla_{\partial_k}\partial_j=\G^i_{jk}\partial_i$\,), then 
$\bigl(\G^i_{kj}\dif\!x^k\bigr)_{i,j}$ are the local connection forms of a connection on $M$. 
\end{proof}  

\begin{prop} \label{prop:proj_equiv} 
Let $\nabla$ and $\widetilde{\nabla}$ be torsion free connections on $M$. Then the following assertions are equivalent:\\ 
\indent 
{\rm (i)} $\nabla$ and $\widetilde{\nabla}$ are projectively equivalent.\\ 
\indent 
{\rm (ii)} There exists a one-form $\a$ on $M$ such that $\widetilde{\nabla}_XY=\nabla_XY+\a(X)Y+\a(Y)X$, 
for any local vector fields $X$ and $Y$ on $M$. 
\end{prop} 
\begin{proof} 
For this we only need the equivalence of the following facts, for a symmetric $(1,2)$ tensor $\G$ on a vector space $V$:\\ 
\indent 
\quad(1) $\G(v,v)$ is proportional to $v$, for any $v\in V$.\\ 
\indent 
\quad(2) There exists $\a\in V^*$ such that $\G(u,v)=\a(u)v+\a(v)u$\,, for any $u,v\in V$.\\ 
\indent 
Indeed, if $\dim V=1$ then this is obvious, whilst, if $\dim V\geq2$ and on assuming (1) then, 
for any $i_1,i_2=1,\ldots,\dim V$, we have $\G^{i_1}_{jk}x^jx^kx^{i_2}=\G^{i_2}_{jk}x^jx^kx^{i_1}$, where $\bigl(x^i\bigr)_i$ is any basis on $V^*$ 
and $x^i\circ \G=\G^i_{jk}x^jx^k$.\\ 
\indent 
Consequently, $\G^i_{jk}=0$ if $j\neq i\neq k$\,. Furthermore, with $i$ fixed, the one-form 
$\a=\frac12\bigl(\G^i_{ii}x^i+2\sum_{j\neq i}\G^i_{ij}x^j\bigr)$ is well defined 
(that is, it does not depend of $i$\,) and satisfies (2)\,, with $u=v$\,.   
\end{proof} 

\indent 
The following definition is, essentially, classical. 

\begin{defn} \label{defn:proj_str} 
A \emph{projective covering} on a manifold $M$ is a family $\bigl\{\nabla^U\bigr\}_{U\in\mathcal{U}}$\,, where:\\ 
\indent 
\quad(a) $\mathcal{U}$ is an open covering of $M$,\\  
\indent 
\quad(b) $\nabla^U$ is a torsion free connection on $U$, for any (nonempty) $U\in\mathcal{U}$,\\ 
\indent 
\quad(c) $\nabla^U$ and $\nabla^V$ are projectively equivalent on $U\cap V$, for any overlapping $U,V\in\mathcal{U}$.\\ 
\indent 
Two projective coverings are \emph{equivalent} if their union is a projective covering. A \emph{projective structure} is an equivalence class 
of projective coverings. 
\end{defn} 

\indent 
For any manifold $M$, endowed with a projective structure, there exists a representative of it $\bigl\{\nabla^U\bigr\}_{U\in\mathcal{U}}$ 
such that, for any $U\in\mathcal{U}$, the connection induced by $\nabla^U$ on $\det(TU)$ is flat; such a representative will be called \emph{special}. 
The existence of special representatives (an essentially known fact) is proved as follows. Let  
$\bigl\{\widetilde{\nabla}^U\bigr\}_{U\in\mathcal{U}}$ be any representative of the projective structure. 
By passing to a refinement of $\mathcal{U}$, if necessary, we may suppose that each $U\in\mathcal{U}$ is the domain of a frame field 
$\bigl(u_U^{\,1},\ldots,u_U^{\,n}\bigr)$ on $M$, over $U$, where $\dim M=n$\,. Let $\a_U$ be the local connection form, 
with respect to $u_U^{\,1}\wedge\ldots\wedge u_U^{\,n}$, of the connection induced by $\widetilde{\nabla}^U$ on $\det(TU)$\,. 
Let $\b_U=-\frac{1}{n+1}\,\a_U$ and $\nabla^U$ be given by $\nabla^U_{\,X}Y=\widetilde{\nabla}^U_{\,X}Y+\b_U(X)Y+\b_U(Y)X$, 
for any $U\in\mathcal{U}$ and any local vector fields $X$ and $Y$ on $U$. Then $\bigl\{\nabla^U\bigr\}_{U\in\mathcal{U}}$ is as required.\\  
 \indent 
Let $\bigl\{\nabla^U\bigr\}_{U\in\mathcal{U}}$ be a representative of a projective structure on $M$. For any overlapping $U,V\in\mathcal{U}$, 
denote by $\a_{UV}$ the one-form on $U\cap V$ which gives $\nabla^V-\nabla^U$, through Proposition \ref{prop:proj_equiv}\,. 
Then $\bigl(\a_{UV}\bigr)_{(U,V)\in\mathcal{U}^*}$ is a cocycle representing, up to a nonzero factor, the obstruction \cite{At-57} to the existence 
of a principal connection on $\det(TM)$\,, where $\mathcal{U}^*=\bigl\{(U,V)\in\mathcal{U}\times\mathcal{U}\,|\,U\cap V\neq\emptyset\bigr\}$\,.  
Recall that this can be defined as the obstruction to the splitting of the following exact sequence of vector bundles 
$$0\longrightarrow M\times\C\longrightarrow E\overset{\r}{\longrightarrow}TM\longrightarrow0\;,$$ 
where $E=\frac{T(\det(TM))}{\C\!\setminus\{0\}}$ and $\r:E\to TM$ is the projection induced by the differential of the projection $\det(TM)\to M$.\\  
\indent 
Let $L$ be a line bundle on $M$. Denote $E=\frac{T(L^*\setminus0)}{\C\!\setminus\{0\}}$\,, and $\r:E\to TM$ the projection.    
Recall that the sheaf of sections of $E$ is given by the sheaf of vector fields on $L^*\setminus0$ which are invariant under the action of $\C\!\setminus\{0\}$\,. 
Therefore to any local sections $s$ and $t$ of $E$ (defined over the same open set of $M$) we can associate their bracket $[s,t]$\,.  
Then $[\cdot,\cdot]$ is skew-symmetric, satisfies the Jacobi identity and $\r$ intertwines it and the usual Lie bracket on local vector fields on $M$.

\begin{rem} \label{rem:ro_power}  
Let $L$ be a line bundle on $M$ and denote $E=\frac{T(L^*\setminus0)}{\C\!\setminus\{0\}}$\,. 
If we replace $L$ by $L^n$, for some $n\in\mathbb{Z}\setminus\{0\}$\,, then in the exact sequence 
$0\longrightarrow M\times\C\overset{\iota}{\longrightarrow}E\overset{\r}{\longrightarrow}TM\longrightarrow0$\,, 
we just need to replace $\iota$ by $(1/n)\,\iota$\,. 
\end{rem} 

\indent 
If $F$ is a vector bundle over $M$ we denote by $\G(F)$ the corresponding sheaf of sections; that is, $\G(U,F)$ is the space of sections of $F$ over $U$, 
for any open set $U\subseteq M$.\\ 
\indent 
The following definition is taken from \cite{Pan-qgfs}\,. 

\begin{defn} \label{defn:ro-connection} 
1) Let $M$ be endowed with a vector bundle $E$, over it, and a morphism of vector bundles $\r:E\to TM$.\\ 
\indent  
If $F$ is a vector bundle over $M$ a \emph{$\r$-connection} on $F$ is a linear sheaf morphism $\nabla:\G(F)\to\G\bigl({\rm Hom}(E,F)\bigr)$ 
such that $\nabla_s(ft)=\r(s)(f)\,t+f\nabla_st$\,, for any local function $f$ on $M$, and any local sections $s$ of $E$ and $t$ of $F$.\\ 
\indent 
2) Suppose (for simplicity) that $\r:E\to TM$ is the projection, with $E=\frac{T(L^*\setminus0)}{\C\!\setminus\{0\}}$ and $L$ a line bundle over $M$. 
Then the \emph{curvature form} of a $\r$-connection $\nabla$ on $F$ is the section $R$ of ${\rm End(F)}\otimes\Lambda^2E^*$  given by 
$R(s_1,s_2)\,t=[\nabla_{s_1},\nabla_{s_2}]\,t-\nabla_{[s_1,s_2]}\,t$\,, for any local sections $s_1,s_2$ of $E$ and $t$ of $F$.\\ 
\indent 
If $\nabla$ is a $\r$-connection on $E$ then its \emph{torsion} is the section $T$ of $E\otimes\Lambda^2E^*$ given by 
$T(s_1,s_2)=\nabla_{s_1}s_2-\nabla_{s_2}s_1-[s_1,s_2]$\,, for any local sections $s_1,s_2$ of $E$. 
\end{defn}  

\begin{rem} \label{rem:ro_inv_classical} 
With the same notations as in Definition \ref{defn:ro-connection}(2)\,, if $L$ admits a (classical) connection then any $\r$-connection on $F$ 
corresponds to pair formed of a $\C\!\setminus\{0\}$ invariant connection on $\p^*F$, and a morphism of vector bundles 
from $M\times\C$ to ${\rm End}(F)$\,, where $\p:L\setminus0\to M$ is the projection. 
\end{rem} 

\indent 
Any (classical) connection $\nabla$ on $F$ defines a $\r$-connection $\widetilde{\nabla}$ given by $\widetilde{\nabla}_st=\nabla_{\r(s)}t$,  
for any local sections $s$ of $E$ and $t$ of $F$.\\ 
\indent 
However, not all $\r$-connections are obtained this way. For example, if a line bundle over $M$ admits a connection then 
its (first) Chern class with complex coefficients is zero, and the converse also holds if $M$ is compact K\"ahler \cite{At-57}\,. 
Nevertheless, any line bundle $L$ is endowed with a \emph{canonical} flat $\r$-connection $\nabla$, where $\r:E\to TM$ is the projection, 
with $E=\frac{T(L^*\setminus0)}{\C\!\setminus\{0\}}$\,. This can be defined as follows. 
Firsly, recall that any local section $s$ of $E$ over an open set $U\subseteq M$ can be seen as a $\C\!\setminus\{0\}$ invariant vector field 
on $L^*\setminus0$\,, whilst any section $t$ of $L$ over $U$ corresponds to a function $f_t$ on $\p^{-1}(U)$\,, where $\p:L^*\setminus0\to M$ is the projection.  
Then, by definition, $\nabla_st=s(f_t)$\,.\\ 
\indent 
For another example, let $V$ be a vector space and let $L$ be the dual of the tautological line bundle over the projective space $PV$. 
From $L^*\setminus0=V\setminus\{0\}$\,, it follows that $\frac{T(L^*\setminus0)}{\C\!\setminus\{0\}}=L\otimes\bigl(PV\times V)$\,. 
Thus, although $PV$ does not admit a connection, we can associate to it the \emph{canonical} flat $\r$-connection given by the tensor product 
of the canonical $\r$-connection of $L$ and the canonical flat connection on $PV\times V$. Note that, the canonical $\r$-connection 
of the projective space is torsion free.\\ 
\indent 
The following fact will be used later on. 

\begin{rem} \label{rem:diff_rational_map} 
Let $L$ be a line bundle over $M$ and let $V$ be a finite dimensional subspace of the space of sections of $L$. 
Then $V$ induces a section $s_V$ of $L\otimes V^*\bigl(={\rm Hom}(M\times V,L)\bigr)$ given by $s_V(x,s)=s_x$\,, 
for any $x\in M$ and $s\in V$. Obviously, the base point set $S_V$ of $V$ is equal to the zero set of $s_V$. 
Assume, for simplicity, that $S_V=\emptyset$\,.\\ 
\indent 
Then the differential of the corresponding map $\phi_V:M\to PV^*$ is induced by $\nabla s_V:E\to L\otimes V^*$, 
where $E=\frac{T(L^*\setminus0)}{\C\!\setminus\{0\}}$ and $\nabla$ is the tensor product 
of the canonical $\r$-connection of $L$ and the canonical flat connection on $M\times V^*$. 
This means that, if we, also, denote by $\dif\!\phi$ the morphism $TM\to\phi^*\bigl(T(PV^*)\bigr)$ corresponding to the differential of $\phi$\,, 
then $\dif\!\phi\circ\r=\r_V\circ(\nabla s_V)$\,, where $\r:E\to TM$ and $\r_V:L\otimes V^*\to\phi^*\bigl(T(PV^*)\bigr)$\,, 
are the projections.  
\end{rem} 

\newpage

\section{The main result on projective structures} 

\indent 
In this section, we prove the following result (cf.\ \cite{Tho-1926}\,, \cite{Rob-proj}\,, \cite{MolMor-1996}\,).  

\begin{thm} \label{thm:proj_ro}  
Let $M$ be a manifold, $\dim M=n\geq2$\,, denote $E=\frac{T(\det(TM))}{\C\!\setminus\{0\}}$ and let $\r:E\to TM$ be the projection.  
There exists a natural correspondence between the following:\\ 
\indent 
{\rm (i)} Projective structures on $M$.\\ 
\indent 
{\rm (ii)} Torsion free $\r$-connections $\nabla$ on $E$ satisfying:\\  
\indent 
\quad{\rm (ii1)} $\nabla_{\1}s=-\frac{1}{n+1}\,s$, for any local section $s$ of $E$, where 
$\1$ is the section of $E$ given by $x\mapsto(x,1)\in M\times\C\!\subseteq E$\,;\\  
\indent 
\quad{\rm (ii2)} The $\r$-connection induced by $\nabla$ on $\Lambda^{n+1}E$ corresponds, under the isomorphism $\Lambda^{n+1}E=\Lambda^n(TM)$\,, 
with the canonical $\r$-connection of $\Lambda^n(TM)$\,;\\ 
\indent 
\quad{\rm (ii3)} $\Ric=0$\,, where $\Ric(s_1,s_2)=\trace\bigl(t\mapsto R(t,s_2)s_1\bigr)$\,, for any $s_1,s_2\in E$, 
with $R$ the curvature form of $\nabla$.  
\end{thm} 
\begin{proof} 
Suppose that $E$ is endowed with a torsion free $\r$-connection $\nabla$ such that, for any local section $s$ of $E$, we have $\nabla_{\1}s=-\frac{1}{n+1}\,s$\,. 
Then, also, $\nabla_s\1=-\frac{1}{n+1}\,s$\,, as $\nabla$ is torsion free and $[\1,s]=0$\,, for any local section $s$ of $E$.\\ 
\indent 
We define the geodesics of $\nabla$ to be those immersed curves $c$ in $M$ for which, locally, there exists a section $s$ of $E$, over $c$\,,  
such that $\r\circ s=\dot{c}$ and $\nabla_ss=0$ (compare \cite[Remark 1.1]{Pan-q_integrab}\,). 
Obviously, this does not depend of the parametrization of $c$ (as an immersion in $M$). Moreover, if $t$ is another section of $E$, over $c$\,, 
such that $\r\circ t=\dot{c}$ then $t=s+f\1$ for some function $f$, on the domain of $c$\,, and, consequently, $\nabla_tt=0$ if and only if $f=0$\,; 
that is, $s=t$\,.\\ 
\indent 
We shall show that for any $x\in M$ and any $X\in T_xM\setminus\{0\}$ there exists a curve $c$ on $M$ and a section $s$ 
of $E$, over $c$\,, such that $\dot{c}(0)=X$, $\r\circ s=\dot{c}$\,, and $\nabla_ss=0$\,; in particular, $c$ is a geodesic (in a neighbourhood of $x$).\\ 
\indent 
For this, let $V$ be the typical fibre of $E$ and let $\bigl(P,M,{\rm GL}(V)\bigr)$ be the frame bundle of $E$\,; denote by $\p:P\to M$ the projection. 
Then $\nabla$ corresponds \cite{Pan-qgfs} to an equivariant map $C:P\times V\to TP$ satisfying: 
\begin{equation} \label{e:principal_ro-connection} 
\begin{split} 
\dif\!\p\bigl(C(u,\xi)\bigr)&=\r(u\xi)\;,\\ 
C(ua,a^{-1}\xi)&=\dif\!R_a\bigl(C(u,\xi)\bigr)\;,\\ 
\end{split} 
\end{equation}  
for any $u\in P$, $a\in{\rm GL}(V)$ and $\xi\in V$, and where $R_a$ is the `(right) translation' on $P$ defined by $a$\,. 
Note that, similarly to the classical case, we have 
\begin{equation} \label{e:covariant_deriv_E} 
\nabla_{u\xi}s=u\,C(u,\xi)(f_s)\;, 
\end{equation} 
for any local section $s$ of $E$, any $u\in P$ such that $\p(u)$ is in the domain of $u$\,, and any $\xi\in V$, 
and where $f_s$ is the equivariant function on $P$ corresponding to $s$.\\ 
\indent 
For $\xi\in V$, we denote \cite{Pan-qgfs} by $C(\xi)$ the vector field on $P$ given by $u\mapsto C(u,\xi)$\,.\\ 
\indent 
Now, let $x\in M$ and $X\in T_xM\setminus\{0\}$\,. Choose $u_0\in P$ and $\xi\in V$ such that $\r(u_0\xi)=X$ and let $c$ be the projection, through $\p$\,, 
of the integral curve $u$ of $C(\xi)$ through $u_0$\,. Thus, if we denote $s=u\xi$\,, then the first relation of \eqref{e:principal_ro-connection} 
implies $\r\circ s=\dif\!\p(\dot{u})=\dot{c}$\,; in particular, $\dot{c}(0)=X$. Furthermore, by \eqref{e:covariant_deriv_E}\,, we have $\nabla_ss=u\,C(u,\xi)(\xi)=0$\,, 
where the second $\xi$ denotes the corresponding constant function along $u$\,.\\  
\indent 
To show that we have constructed, indeed, a projective structure, let $c_U:TU\to E|_U$ be the local section of $\r$ corresponding to a connection on 
$\det(TU)$\,, for some open set $U\subseteq M$ (note that, we may cover $M$ with such open sets $U$). 
Then $E|_U=TU\oplus(U\times\C\!)$\,, where we have identified $TU$ and the image of $c$\,; in particular, $\r|_U$ is just the projection 
from $E|_U$ onto $TU$. Let $\nabla^U$ be the (torsion free) connection on $U$ given by $\nabla^U_{\,X}Y=\r(\nabla_XY)$\,, for any 
local vector fields $X$ and $Y$ on $U$. Then if we intersect with $U$ any geodesic of $\nabla$ we obtain a geodesic of the projective structure on $U$, 
determined by $\nabla^U$.\\ 
\indent 
We have, thus, proved that any torsion free $\r$-connection $\nabla$ on $E$, satisfying the condition $\nabla_{\1}s=-\frac{1}{n+1}\,s$, 
for any local section $s$ of $E$, determines a projective structure on $M$.\\ 
\indent 
Conversely, suppose that $M$ is endowed with a projective structure given by the special projective covering $\bigl\{\nabla^U\bigr\}_{U\in\mathcal{U}}$\,.\\ 
\indent 
As $\nabla^U$ induces a flat connection on $\det(TU)$\,, it corresponds to a section $c_U$\,, over $U$, of $\r$\,; furthermore, 
$c\circ[X,Y]=[c\circ X,c\circ Y]$ for any local vector fields $X$ and $Y$ on $U$. Therefore there exists a unique $\b_U\in\G\bigl(E^*|_U\bigr)$ such that,  
for any $t\in E|_U$\,, we have $t=c_U\bigl(\r(t)\bigr)+\b_U(t)\1$\,.\\ 
\indent 
Let $U,V\in\mathcal{U}$\,, be such that $U\cap V\neq\emptyset$\,, and let $\a_{UV}$ be the one-form on $U\cap V$ given by Proposition \ref{prop:proj_equiv} 
applied to $\nabla^U|_{U\cap V}$ and $\nabla^V|_{U\cap V}$\,. Then, on $U\cap V$, we have $c_V=c_U-(n+1)\a_{UV}\1$\,; equivalently, 
$(n+1)\a_{UV}\bigl(\r(t)\bigr)=\b_V(t)-\b_U(t)$\,, for any $t\in U\cap V$.\\ 
\indent 
For any $U\in\mathcal{U}$, we define a $\r$-connection $\widetilde{\nabla}^U$ on $E|_U$ by 
$$\widetilde{\nabla}^U_{\,s}t=c_U\bigl(\nabla^U_{\r(s)}\bigl(\r(t)\bigr)-\tfrac{1}{n+1}\b_U(s)\r(t)-\tfrac{1}{n+1}\b_U(t)\r(s)\bigr) 
+\bigl(b_U(s,t)+\r(s)\bigl(\b_U(t)\bigr)\bigr)\1\;,$$ 
for any local sections $s$ and $t$ of $E|_U$\,, where $b_U$ is some section of $\odot^2E^*|_U$\,; consequently, 
\begin{equation*} 
\begin{split} 
\widetilde{\nabla}^U_{\,s}t-\widetilde{\nabla}^U_{\,t}s&=c_U\bigl([\r(s),\r(t)]\bigr)+\bigl(\r(s)\bigl(\b_U(t)\bigr)-\r(t)\bigl(\b_U(s)\bigr)\bigr)\1\\ 
&=\bigl[c_U\bigl(\r(s)\bigr),c_U\bigl(\r(t)\bigr)\bigr]+\bigl(\r(s)\bigl(\b_U(t)\bigr)-\r(t)\bigl(\b_U(s)\bigr)\bigr)\1=[s,t]\;,  
\end{split} 
\end{equation*} 
that is, $\widetilde{\nabla}^U$ is torsion free.\\ 
\indent 
Let $U\in\mathcal{U}$, and denote by $\Ric^U$ the Ricci tensor of $\nabla^U$ defined by $\Ric^U(X,Y)=\trace\bigl(Z\mapsto R^U(Z,Y)X\bigr)$\,, for any $X,Y\in TM$, 
where $R^U$ is the curvature form of $\nabla$. For $s,t\in E|_U$, we define 
$$b_U(s,t)=\tfrac{n+1}{n-1}\Ric^U\bigl(\r(s),\r(t)\bigr)-\tfrac{1}{n+1}\b_U(s)\b_U(t)\;.$$ 
Then a straightforward computation shows that $\widetilde{\nabla}^U|_{U\cap V}=\widetilde{\nabla}^V|_{U\cap V}$\,, for any 
$U,V\in\mathcal{U}$\,, with $U\cap V\neq\emptyset$\,. We have, thus, obtained a torsion free $\r$-connection $\nabla$ on $E$ 
which it is easy to prove that it satisfies (ii1)\,.\\ 
\indent 
Further, we may suppose that, for any $U\in\mathcal{U}$, there exists an $n$-form $\o_U$ on $U$ such that $\nabla^U\!\o_U=0$\,. 
Consequently, $\nabla_t\bigl(\r^*\o_U\bigr)=\frac{n}{n+1}\b_U(t)\,\r^*\o_U$\,, for any $t\in E|_U$\,.\\ 
\indent  
Note that, the isomorphism $\Lambda^n(T^*U)=\Lambda^{n+1}\bigl(E^*|_U\bigr)$ is expressed by $\o_U\mapsto\a\wedge\r^*\o_U$\,, where 
$\a$ is any local section of $E^*|_U$ such that $\a(\1)=1$\,. Also, (ii1) implies that, for any $t\in E|_U$\,, we have 
$\nabla_t\bigl(\a\wedge\r^*\o_U\bigr)=\b_U(t)\,\a\wedge\r^*\o_U$\,.\\ 
\indent 
On the other hand, the relation $t=c_U\bigl(\r(t)\bigr)+\b_U(t)\1$\,, for any $t\in E|_U$\,, means that $\b_U$ is the `difference' between 
the connection induced by $\nabla^U$ on $\Lambda^n(TU)$ and the canonical $\r$-connection $\overset{\rm can}{\nabla}$ on $\Lambda^n(TU)$\,;  
equivalently, $\overset{\rm can}{\nabla}_{\!t\,}\o_U=\b_U(t)\,\o_U$\,, for any $t\in E|_U$\,. Thus, $\nabla$ satisfies (ii2)\,.\\ 
\indent 
Finally, let $R$ be the curvature form of $\nabla$. Then a straightforward calculation shows that, on each $U\in\mathcal{U}$, we have 
\begin{equation} \label{e:proj_curvature} 
R=c_U\bigl(\r^*W^U\bigr)+\tfrac{n+1}{n-1}\bigl(\r^*C^U\bigr)\,\1\;, 
\end{equation} 
where $W^U$ and $C^U$ are the projective Weyl and Cotton-York tensors of $\nabla^U$, respectively, given by 
\begin{equation*} 
\begin{split} 
W^U(X,Y)Z&=R^U(X,Y)Z+\tfrac{1}{n-1}\bigl(\Ric^U(X,Z)Y-\Ric^U(Y,Z)X\bigr)\;,\\  
C^U(X,Y,Z)&=\bigl(\nabla^U_X\Ric^U\bigr)(Y,Z)-\bigl(\nabla^U_Y\Ric^U\bigr)(X,Z)\;, 
\end{split} 
\end{equation*} 
for any $X,Y,Z\in TU$. As \eqref{e:proj_curvature} implies (ii3)\,, the proof is complete.  
\end{proof} 

\begin{rem} \label{rem:proj_ro}
1) Suppose that, in Theorem \ref{thm:proj_ro}\,, there exists a line bundle $L$ such that $L^{n+1}=\Lambda^n(TM)$\,. Then we may replace 
$\det(TM)$ by $L^*\setminus0$\,, and, by Remark \ref{rem:ro_power}\,,  condition (ii1) becomes $\nabla_{\1}s=s$, for any local section $s$ of $E$, 
as satisfied by the canonical $\r$-connection of the projective space. Furthermore, the canonical $\r$-connection of the projective space, also, 
satisfies (ii2) and (ii3)\,, and the corresponding geodesics are the projective lines (as the `the second fundamental form', 
with respect to the canonical $\r$-connection, of any projective subspace is zero).\\  
\indent 
2) Condition (ii1) of Theorem \ref{thm:proj_ro} is necessary for $\nabla$ to be able to define a projective structure. Further, condition (ii2) 
fixes the `horizontal' part $\r\circ\nabla$ of $\nabla$ (among the torsion free $\r$-connections satisfying (ii1)\,), 
whilst (ii3) fixes the `vertical' part $\b_U\circ\nabla$, for $U\in\mathcal{U}$\,.       
\end{rem}

\section{Applications} 

\indent 
In this section, firstly, we explain how the well known characterisation of `projective flatness' can be improved by using our approach. 

\begin{cor} \label{cor:proj_ro_flat} 
Let $M$ be endowed with a projective structure, given by the torsion free $\r$-connection $\nabla$, and  
suppose that there exists a line bundle $L$ over $M$ such that $L^{n+1}=\Lambda^n(TM)$\,, where $\dim M=n\geq2$\,.\\ 
\indent 
Then $\nabla$ is flat if and only if there exists a (globally defined) local diffeomorphism from a covering space of $M$ to $\C\!P^n$ mapping 
the geodesics into projective lines.  
\end{cor} 
\begin{proof} 
Assume, for simplicity, $M$ simply-connected. Also, by Remark \ref{rem:proj_ro}(1)\,, we may suppose that 
$E=\frac{T(L^*\setminus0)}{\C\!\setminus\{0\}}$ so that $\nabla_{\1}s=s$, for any local section $s$ of $E$. 
Then, on denoting by $V$ the typical fibre of $E$, we have that $\nabla$ is flat if and only if $L\setminus0$ is a reduction to $\C\!\setminus\{0\}$ 
of the frame bundle of $E$, where $\C\!\setminus\{0\}\subseteq{\rm GL}(V)$ through $\l\mapsto\l\,{\rm Id}_V$. 
Equivalently, $\nabla$ is flat if and only if there exists an isomorphism of vector bundles $\a:E\to L\otimes V$, preserving the $\r$-connections.  
In particular, if we define $s=\a(\1)$ then $s$ is a section of $L\otimes V$ which is nowhere zero; note, also, that $\nabla s=\a$\,. 
Therefore $s$ induces a section of $P(L\otimes V)=M\times PV$ given by $x\mapsto\bigl(x,\phi(x)\bigr)$\,, for any $x\in M$, 
for some map $\phi:M\to PV$. Moreover, $\phi$ is as required, as, by Remark \ref{rem:diff_rational_map}\,, its differential is induced by $\a$\,.  
The proof is complete.   
\end{proof} 

\indent 
Recall that a rational curve on a manifold $M$ is a nonconstant map from the projective line to $M$. 
Also, $\ol(n)$ denotes the line bundle of Chern number $n\in\mathbb{Z}$ over the projective line. 

\begin{cor} \label{cor:proj_ro_app} 
Let $M$ be a manifold endowed with a projective structure and a smooth rational curve $t\subseteq M$ with normal bundle $k\ol(1)\oplus(n-k-1)\ol$\,, 
where $\dim M=n\geq2$ and $k\in\{1,\ldots,n-1\}$\,.\\ 
\indent 
Then $t$ is a geodesic, $k=n-1$ and the projective structure of $M$ is flat. 
\end{cor}      
\begin{proof} 
By using \cite{Kod}\,, and by passing to an open neighbourhood of $t$\,, if necessary, we may assume that $M$ is covered with a locally complete 
$(n+k-1)$-dimensional family of smooth rational curves each of which has normal bundle $k\ol(1)\oplus(n-k-1)\ol$.\\  
\indent  
Denote $E=\frac{T(\det(TM))}{\C\!\setminus\{0\}}$ and let $\r:E\to TM$ be the projection. We have an exact sequence 
$$0\longrightarrow\ol\longrightarrow E|_t\overset{\r|_t}{\longrightarrow}\ol(2)\oplus k\ol(1)\oplus(n-k-1)\ol\longrightarrow0\;.$$ 
This exact sequence corresponds to $k+2\in\C\!=H^1\bigl(t\,,\ol(-2)\oplus k\ol(-1)\oplus(n-k-1)\ol\bigr)$ (the Chern number of $\ol(k+2)=\det(TM)|_t$\,), 
and, consequently, we must have $E|_t=(k+2)\ol(1)\oplus(n-k-1)\ol$\,.\\ 
\indent 
Let $\nabla$ be the $\r$-connection on $E$ giving the projective structure of $M$. The second fundamental form of $t$\,, with respect to $\nabla$,  
is a section of $$\bigl(2\ol(1)\bigr)^*\otimes\bigl(2\ol(1)\bigr)^*\otimes\bigl(k\ol(1)\oplus(n-k-1)\ol\bigr)=4k\ol(-1)\oplus4(n-k-1)\ol(-2)$$ 
and therefore it is zero. Thus, $t$ is a geodesic and it follows that $k=n-1$ (as the space of rational geodesics has dimension $2n-2$).\\ 
\indent 
Let $R$ be the curvature form of $\nabla$ and note that we can see it as a section of $E\otimes\bigotimes^3E^*$. 
Then the restriction of $R$ to any smooth rational curve, with normal bundle $(n-1)\ol(1)$\,, 
is a section of $(n+1)\ol(1)\otimes\bigotimes^3\bigl((n+1)\ol(-1)\bigr)=(n+1)^4\ol(-2)$ and therefore it is zero. 
Consequently, $R=0$ and the proof is complete. 
\end{proof} 

\indent 
The first application of Corrolaries \ref{cor:proj_ro_flat} and \ref{cor:proj_ro_app} is that if the twistor space of a quaternionic manifold $P$ 
is endowed with a complex projective structure then $P$ can be locally identified, through quaternionic diffeomorphisms, with the quaternionic projective space.\\ 
\indent 
Also (compare \cite{Bel-null_geods}\,), any projective structure that admits a rational geodesic must be flat.\\ 
\indent 
Finally, as any Fano manifold is compact simply-connected and admits rational curves as in Corollary \ref{cor:proj_ro_app} (see \cite{Paltin-2005} 
and the references therein) from Corrolary \ref{cor:proj_ro_flat} we obtain the following fact \cite{BisMcK-holo_Cartan}\,: 
the projective space is the only Fano manifold which admits a projective structure 
(compare \cite[(5.3)]{KobOch-1980}\,, \cite{HwaMok-1997}\,, \cite{Paltin-2005}\,).

\end{document}